\documentclass[11pt, a4paper, reqno]{amsart}

\usepackage{a4wide}
\usepackage{amssymb, amsmath}

\usepackage[]{hyperref}
\hypersetup{
    pdfpagemode=UseNone,
}

\usepackage[latin1]{inputenc}
\usepackage{subfigure}

\newtheorem{thm}{Theorem}
\newtheorem{cor}[thm]{Corollary}

\newtheorem{defn}{Definition}

\newcommand{\abs}[1]{\left|#1\right|}
\newcommand{\set}[1]{\left\{#1\right\}}

\DeclareMathOperator{\comp}{comp}
\DeclareMathOperator{\anum}{a}
\DeclareMathOperator{\bnum}{b}
\DeclareMathOperator{\cnum}{c}
\DeclareMathOperator{\snum}{sa}
\DeclareMathOperator{\sa}{sa}
\DeclareMathOperator{\sech}{sech}
\DeclareMathOperator{\Cat}{Cat}

\def\ds{\displaystyle}

\begin{document}

\title{Signed a-polynomials of graphs and Poincar\'e polynomials of real toric manifolds}

\author{Seunghyun Seo}
\address[Seunghyun Seo]{Department of Mathematics Education, Kangwon National University, Chuncheon 200-701, Korea}
\email{shyunseo@kangwon.ac.kr}

\author{Heesung Shin}
\address[Heesung Shin]{Department of Mathematics, Inha University, Incheon 402-751, Korea}
\email{shin@inha.ac.kr}
\date{\today}                                           



\maketitle

\begin{abstract}
Recently, Choi and Park introduced an invariant of a finite simple graph, called \emph{signed a-number}, arising from computing certain topological invariants of some specific kinds of real toric manifolds.
They also found the signed a-numbers of path graphs, cycle graphs, complete graphs, and star graphs.

We introduce a \emph{signed a-polynomial} which is a generalization of the signed a-number and gives \emph{a-, b-, and c-numbers}.
The signed a-polynomial of a graph $G$ is related to the Poincar\'e polynomial $P_{M(G)}(z)$, which is the generating function for the Betti numbers of the real toric manifold $M(G)$. 
We give the generating functions for the signed a-polynomials of not only path graphs, cycle graphs, complete graphs, and star graphs, but also complete bipartite graphs and complete multipartite graphs.
As a consequence, we find the Euler characteristic number and the Betti numbers of the real toric manifold $M(G)$ for complete multipartite graphs $G$.

\end{abstract}

\section{Introduction}
In algebraic topology, Choi and Park \cite{CP12} recently introduced a graph invariant, called \emph{signed a-number} of finite simple graphs $G$, denote by $\snum(G)$ as follows:
\begin{itemize}
\item $\snum(\emptyset) = 1$.
\item $\snum(G)$ is the product of signed a-numbers of connected components of $G$.
\item $\snum(G) = 0$ if $G$ is a connected graph with odd vertices.
\item If $G$ is connected with even vertices, then $\snum(G)$ is given by the negative of the sum of signed a-numbers of all induced subgraphs $G'$ of $G$ except itself $G$.
\end{itemize}
Let the \emph{a-number} $\anum(G)$ be the absolute value of the signed a-number of $G$,
the \emph{b-number} $\bnum(G)$ the  sum of  signed a-numbers induced subgraphs of $G$, and
the \emph{c-numbers} $\cnum_i(G)$ the sum of a-numbers of induced subgraphs of $G$ with $i$ vertices.

These numbers arise from computing certain \emph{topological invariants} of some specific kinds of real toric manifolds which are important objects in toric topology.
For a finite simple graph $G$, the real toric manifold $M(G)$ is the set of real points in the toric manifold associated to the graph associahedron $P_{\mathcal{B}(G)}$ which is the  \emph{nestohedron} as the Minkowski sum of simplices obtained from connected induced subgraphs of $G$.  
For further information, see \cite{Del88, DJ91, Pos09, PRW08}.

Recently, Choi and Park \cite[Theorem 1.1]{CP12} showed that
the Euler characteristic $\chi(M(G))$ is equal to $\bnum(G)$ and
the (rational) Betti number $\beta_i(M(G))$ is equal to $\cnum_{2i}(G)$.
We remark that $\cnum_{2i}(G)$ is the same with $\anum_i(G)$ in \cite{CP12}.
They also computed these numbers of path graphs $P_{2n}$, cycle graphs $C_{2n}$,  complete graphs $K_{2n}$, and star graphs $K_{1,2n-1}$. 

In this paper, we introduce a \emph{signed a-polynomial} which is a generalization of the signed a-number and gives a-, b-, and c-numbers.
The signed a-polynomial of a graph $G$ is related to the Poincar\'e polynomial $P_{M(G)}(z)$, which is the generating function for the Betti numbers of the real toric manifold $M(G)$. 
The relation will be shown in the equation \eqref{eq:convert}.
We give the signed a-polynomials of not only path graphs, cycle graphs, complete graphs,  and star graphs, but also complete bipartite graphs $K_{p,q}$ and complete multipartite graphs $K_{p_1, \dots, p_m}$.
As a consequence, we find $\chi(M(G))$ and $\beta_i(M(G))$ for $G=K_{p,q}$ and $G=K_{p_1, \dots, p_m}$.


\section{Preliminaries}

From now on, we assume that a graph is finite, undirected, and simple.
We rewrite a formal definition of a signed a-number $\snum(G)$ of a graph $G=(V,E)$ in the previous section by
$$
\snum(G) =
\begin{cases}
1 & \text{if $G$ is the empty graph;}\\
0 & \text{if $G$ is connected and $\abs{V}$ is odd;}\\
\ds - \sum_{V' \subsetneq V} \snum(G|_{V'}) & \text{if $G$ is connected and $\abs{V}$ is even $\ge$ 2;}\\
\ds \prod_{G' \in \comp(G)} \snum(G') & \text{if $G$ is disconnected,}
\end{cases}
$$
where $G|_{V'}$ is the induced subgraph of $G$ by a vertex subset  $V'$ and $\comp(G)$ is the set of connected components of $G$.
From above definition, it is easy to check $\snum(G)=0$ for every graph $G$ with at least one connected component on odd vertices; and $\sum_{V' \subseteq V} \snum(G|_{V'}) = 0 $ for every nonempty graph $G$ on $V$ with every connected component on even vertices.
Thus,  we find a simpler equivalent  definition of a signed a-number as follows.
\begin{defn}
A \emph{signed a-number} $\snum(G)$ of a graph $G=(V,E)$ is defined by
\begin{align}\label{eq:snum}
\snum(G) =
\begin{cases}
1, & \text{if $G$ is the empty graph;}\\
0, & \text{if $G$ has a connected component on odd vertices;}\\
\ds - \sum_{V' \subsetneq V} \snum(G|_{V'}), & \text{otherwise.}\\
\end{cases}
\end{align}
\end{defn}
Consequently, we define a-, b-, and c-numbers of a graph with the signed a-numbers.
\begin{defn}
The \emph{a-, b-, and c-numbers} of a graph $G$, denoted by $\anum(G)$, $\bnum(G)$, and $\cnum_i(G)$, are defined by
\begin{align}
\anum(G) &= (-1)^{\abs{V}/2} \snum(G), \label{eq:anum}\\
\bnum(G) &= \sum_{V' \subseteq V} \snum(G|_{V'}), \label{eq:bnum}\\
\cnum_{i}(G) &= \sum_{V' \subseteq V \atop \abs{V'} = i} \anum(G|_{V'}) =  (-1)^{i/2} \sum_{V' \subset V \atop \abs{V'} = i} \snum(G|_{V'}). \label{eq:cnum}
\end{align}
\end{defn}
By definition, for any graph $G$, it hold that $\cnum_i(G) = 0$ if $i$ is odd, and $\cnum_{n}(G) = \anum(G)$ if  $n$ is the number of vertices of $G$.
In topological viewpoint \cite[Remark 2.2]{CP12}, it is obvious that $\anum(G)$ and $\cnum_{i}(G)$ are nonnegative integers.

\section{On signed a-polynomials}
Now, we introduce a generalization of a-, b-, and c-numbers of graphs.
\begin{defn}[Signed a-polynomial]
The \emph{signed a-polynomial} $\snum(G;t)$ of a graph $G$ is defined by
\begin{align}\label{eq:sagt}
\snum(G;t) = \sum_{V' \subseteq V(G)} \snum(G|_{V'}) \ t^{\abs{V \setminus V'}}.
\end{align}
\end{defn}
From the equations \eqref{eq:snum} -- \eqref{eq:sagt}, for $\abs{V(G)} = n$, it holds that
\begin{align*}
\snum(G) &= \snum(G;0),& \anum(G) &  = (-1)^{n/2} \snum(G;0),\\
\bnum(G) &= \snum(G;1), & \cnum_i(G) &= (-1)^{i/2} [t^{n-i}] \snum(G;t).
\end{align*}
Thus, $\snum(G;t)$ is represented as the sum of $\cnum_{i}(G)$'s by
\begin{align}
\snum(G;t) = \sum_{j=0}^{\lfloor n/2 \rfloor} (-1)^{j} c_{2j}(G) t^{n-2j}.
\label{eq:formula}
\end{align}
For example, if $G = (\set{A,B,C,D}, \set{\set{A,B}, \set{A,C}, \set{A,D}, \set{B,C}, \set{B,D}})$, then $$\sa(G;t) = t^4 -5t^2 +4.$$
Thus, $\snum(G) = \anum (G)= 4$, $\bnum(G) = 0$, and $\set{\cnum_{i}(G)}_{i=0}^4 = 1,0,5,0,4.$

\begin{rmk}
The \emph{Poincar\'e polynomial} $P_{M(G)}(z)= \sum_{i \ge 0} \beta_i(M(G)) z^i$
is the generating function for the \emph{Betti numbers} $\beta_i(M(G))$ of the real toric manifold $M(G)$.
Since $\beta_i(M(G)) = \cnum_{2i}(G)$ in \cite[Theorem 1.1]{CP12},  it holds that
\begin{align}
P_{M(G)} (z) &= (\sqrt{-z})^{\abs{V}} \sa\left(G; \frac{1}{\sqrt{-z}}\right). \label{eq:convert}
\end{align}

\end{rmk}

In the rest of the section, we compute the generating functions for signed a-polynomials of path graphs, cycle graphs, complete graphs, and star graphs.

\begin{thm}
Let $P_n$ be the \emph{path graph} with $n$ vertices, which is a tree with exactly $n-2$ vertices of degree 2.  Then the generating function for signed a-polynomials of $P_n$ is given by
\begin{align}
\sum_{n\ge 0} \snum(P_{n}; t) {x^{n}} &= \frac{-1+2tx +\sqrt{1+4x^2}}{2tx - 2(t^2-1)x^2}.
\label{eq:gfpn}
\end{align}
\end{thm}

\begin{proof}
From Theorem 2.5 in \cite{CP12}, it is known that
$$\cnum_{2i} (P_n) =  \binom{n}{i} - \binom{n}{i-1} = \Cat_{n-i,i},$$
with Catalan triangle numbers $\Cat_{n,k} = \binom{n+k}{k} - \binom{n+k}{k-1}$.
Using the formula \eqref{eq:formula}, we have
\begin{align*}
\snum(P_n; t) = \sum_{j=0}^{\lfloor n/2 \rfloor} (-1)^j \Cat_{n-j,j}  t^{n-2j}.
\end{align*}
Thus, we obtain
\begin{align}
\sum_{n\ge 0} \snum(P_{n}; t) {x^{n}}
= \sum_{n\ge 0} \sum_{j=0}^{\lfloor n/2 \rfloor} (-1)^j \Cat_{n-j,j} t^{n-2j} {x^{n}}
= \sum_{k\ge 0} \sum_{j\ge 0}  \Cat_{k, j} {(-x/t)^{j}} (tx)^{k}.
\label{eq:gfP}
\end{align}
Since the generating function for Catalan triangle numbers is
$$\sum_{n\ge 0} \sum_{i\ge 0}  \Cat_{n, i} {w^i} z^{n} = \frac{\Cat(wz)}{1-z\Cat(wz)},$$
where $\Cat(x) = \frac{1-\sqrt{1-4x}}{2x}$,
therefore the formula \eqref{eq:gfP} becomes the formula \eqref{eq:gfpn}.
\end{proof}

\begin{table}[t]
\centering
\begin{tabular}{c||c|c|c}
$G$ & $P_0$ & $P_{2n}$ & $P_{2n+1}$ \\
\hline
$\snum(G)$ & 1 & $(-1)^n \Cat_{n}$ & $0$ \\
$\anum(G)$ & 1 & $\Cat_{n}$ & $0$\\
$\bnum(G)$ & 1 & $0$ & $(-1)^n \Cat_{n}$ \\
$\cnum_{2i}(G)$ & $\delta_{i,0}$ & $\ds \Cat_{2n-i,i}$ & $\ds \Cat_{2n+1-i,i}$\\
\hline
\\[-9pt]
g.f. for $\snum(G;t)$ & \multicolumn{3}{c}{$\ds \sum_{n\ge 0} \snum(P_{n}; t) {x^{n}} = \frac{-1+2tx +\sqrt{1+4x^2}}{2tx - 2(t^2-1)x^2}$}
\end{tabular}
\caption{Numbers for path graphs $P_n$, where Catalan triangle numbers $\Cat_{n,k} = \binom{n+k}{k} - \binom{n+k}{k-1}$ and Catalan numbers $\Cat_{n} = \Cat_{n,n}=\frac{1}{n+1} \binom{2n}{n}$.}
\end{table}

\begin{rmk}
For two given sequences $\sigma = (s_0, s_1, s_2, \ldots)$ and $\tau = (t_1, t_2, t_3, \ldots)$, define the generalized Catalan number $B_n$ by the sum of weighted Motzkin paths from $(0,0)$ to $(n,0)$ with up steps $(1,1)$, horizontal steps $(1,0)$, and down steps $(1,-1)$ where we associate weight $1$ to each up step, weight $s_k$ to each horizontal step on the line $y=k$, and weight $t_k$ to each down step between two lines $y=k-1$ and $y=k$.
For example, if $\sigma \equiv 0$ and $\tau \equiv 1$, then $B_{2n} = \Cat_n$.
In Section 7.4 in \cite{Aig07},
the generating function $B(z) = \sum_{n\ge 0} B_n z^n$ of the generalized Catalan number $B_n$ with $\sigma = (a, s, s, \ldots)$ and $\tau= (b, u, u, \ldots)$ is equal to
\begin{equation}
B(z) = \frac{ (2u -b) + (bs-2au) z - b \sqrt{1-2sx + (s^2-4u) z^2}}{2(u-b) + 2(bs-2au+ab)z +2(a^2 u -abs +b^2)z^2}.
\label{eq:gfgc}
\end{equation}

For $(a,s,b,u,z) = (t,0,-1,-1,x)$, the formula \eqref{eq:gfgc} gives a combinatorial interpretation of the following formula
$$\sum_{n\ge 0} \snum(P_{n}; t) {x^{n}} = \frac{-1+2tx +\sqrt{1+4x^2}}{2tx - 2(t^2-1)x^2}$$
and for $(a,s,b,u,z) = (0,0,-1,t^2-1,x)$, the formula \eqref{eq:gfgc} gives a combinatorial interpretation of the following formula
$$\sum_{n\ge 0} \snum(P_{2n}; t) {x^{2n}}  = \frac{-(t^2+1) - (t^2-1) \sqrt{1+4x^2}}{-2t^2 + 2(t^2-1)^2x^2}.$$
\end{rmk}

\begin{thm}
Let $C_n$ be the \emph{cycle graph} with $n$ vertices, which is a connected graph with all vertices of degree 2.  Then the generating function for signed a-polynomials of $C_n$ is given by
\begin{align}
\sum_{n\ge 0} \snum(C_{n}; t) x^{n}
&=  \frac{1}{2} + \frac{1}{2\sqrt{1+4x^2}} \cdot \frac{(t^2+1) x + t\sqrt{1+4x^2}}{t-(t^2-1)x}.
\label{eq:gfcn}
\end{align}
\end{thm}

\begin{proof}
From Theorem 2.6 in \cite{CP12}, it is known that
$$\cnum_{2i} (C_n) =
\begin{cases}
1, & \text{if $i=n=0$,}\\
\frac{1}{2}\binom{n}{n/2}, &\text{if $2i = n > 0$,} \\
\binom{n}{i}, & \text{if $2i < n$.}
\end{cases}
$$
Using the formula \eqref{eq:formula}, we obtain
\begin{align}
\sum_{n\ge 0} \snum(C_{n}; t) {x^{n}}
&= \sum_{n\ge 0} \sum_{j=0}^{\lfloor n/2 \rfloor} (-1)^j   \cnum_{2j}(C_n) \ t^{n-2j} {x^{n}}
= \sum_{k\ge 0} \sum_{j\ge 0} (-1)^j   \cnum_{2j}(C_{2j+k}) \ t^{k} {x^{2j+k}} \notag \\
&= \frac{1}{2} -\frac{1}{2} \sum_{j\ge 0}  \binom{2j}{j}  {(- x^2)^{j}} +  \sum_{k\ge 0}  \sum_{j\ge 0} \binom{2j+k}{j} \ (tx)^{k}   {(-x^2)^{j}} .\label{eq:gfC}
\end{align}
From $\sum_{n \ge 0} \binom{2n+k}{n} z^n = \frac{1}{1-\sqrt{1-4z}} \left( \frac{1-\sqrt{1-4z}}{2z}\right)^k$, we have two generating functions:
\begin{align*}
\sum_{n \ge 0} \binom{2n}{n} z^n &= \frac{1}{1-\sqrt{1-4z}}, \\
\sum_{n \ge 0} \sum_{k\ge 0} \binom{2n+k}{n} w^k z^n &=  \frac{1}{1-\sqrt{1-4z}}\cdot\frac{1}{ 1 - w \left( \frac{1-\sqrt{1-4z}}{2z}\right)} .
\end{align*}
Using above two generating functions, the formula \eqref{eq:gfC} becomes the formula \eqref{eq:gfcn}.
\end{proof}

\begin{table}[t]
\centering
\begin{tabular}{c||c|c|c}
$G$ & $C_0$ & $C_{2n}$  & $C_{2n+1}$ \\
\hline
$\snum(G)$ & 1 & $ \ds \frac{(-1)^n}{2} \binom{2n}{n} $ & $0$ \\
$\anum(G)$ & 1 & $\ds \frac{1}{2} \binom{2n}{n} $ & $0$\\
$\bnum(G)$ & 1 & $0$ & $\ds (-1)^n \binom{2n}{n}$ \\
$\cnum_{2i}(G)$ & $\delta_{i,0}$ & $\ds
\begin{cases}
\frac{1}{2} \binom{2n}{n} ,& \text{if $i=n$}\\
\binom{2n}{i} ,& \text{if $i<n$}
\end{cases}
$ & $\ds \binom{2n+1}{i}$\\
\hline
\\[-9pt]
g.f. for $\snum(G;t)$ & \multicolumn{3}{c}{$\ds
\sum_{n\ge 0} \snum(C_{n}; t) x^{n}
=  \frac{1}{2} + \frac{1}{2\sqrt{1+4x^2}} \cdot \frac{(t^2+1) x + t\sqrt{1+4x^2}}{t-(t^2-1)x} $}
\end{tabular}
\caption{Numbers for cycle graphs $C_n$.}
\end{table}

Let $A_n$ be the \emph{$n$-th Euler zigzag number} for which the exponential generating function is
\begin{align}
\sum_{n\ge 0} A_n \frac{z^n}{n!} = \sec z + \tan z.
\label{eq:zigzag}
\end{align}

\begin{thm}
Let $K_n$ be the \emph{complete graph} with $n$ vertices. Then the exponential generating function for signed a-polynomials of $K_n$ is given by
\begin{align}
\sum_{n\ge 0} \snum(K_{n}; t) \frac{x^n}{n!} &= e^{tx} \sech x.
\label{eq:gfkn}
\end{align}
\end{thm}

\begin{proof}
From Theorem 2.8 in \cite{CP12}, it is known that
$$\snum (K_{2n}) =  (-1)^{n} A_{2n}.$$
Using the formula \eqref{eq:sagt}, we obtain
\begin{align}
\sum_{n\ge 0} \snum(K_{n}; t) \frac{x^n}{n!}
&= \sum_{n\ge 0} \sum_{j=0}^{\lfloor n/2 \rfloor} \binom{n}{2j} \snum(K_{2j}) \ t^{n-2j} \frac{x^n}{n!} \notag \\
&= \sum_{k\ge 0} \sum_{j\ge 0} \binom{k+2j}{2j} (-1)^{j} A_{2j} \ t^{k} \frac{x^{k+2j}}{(k+2j)!} \notag \\
&= \left( \sum_{k\ge 0} \frac{(tx)^{k}}{k!} \right) \left( \sum_{j\ge 0}  A_{2j} \frac{(\imath x)^{2j}}{(2j)!} \right) \label{eq:gfK}
\end{align}
By \eqref{eq:zigzag}, the formula \eqref{eq:gfK} becomes the formula \eqref{eq:gfkn}.
\end{proof}

\begin{table}[t]
\centering
\begin{tabular}{c||c|c|c}
$G$ & $K_0$ & $K_{2n}$  & $K_{2n+1}$ \\
\hline
$\snum(G)$ & 1 & $(-1)^n A_{2n}$ & $0$ \\
$\anum(G)$ & 1 & $A_{2n}$ & $0$\\
$\bnum(G)$ & 1 & $0$ & $(-1)^n A_{2n+1}$ \\
$\cnum_{2i}(G)$ & $\delta_{i,0}$ & $\ds \binom{2n}{2i} A_{2i}$ & $\ds \binom{2n+1}{2i} A_{2i}$\\
\hline
\\[-9pt]
e.g.f. for $\snum(G;t)$ & \multicolumn{3}{c}{$\ds \sum_{n\ge 0} \snum(K_{n}; t) \frac{x^n}{n!} = e^{tx} \sech(x)$}
\end{tabular}
\caption{Numbers for complete graphs $K_n$, where $\sum_{n \ge 0} A_n \frac{z^n}{n!}=\sec  z + \tan z$.}
\end{table}

\begin{rmk}
The Euler polynomials $E_n(t)$ is defined by the exponential generating function
$\sum_{n\ge 0} E_n\left(t \right) \frac{x^n}{n!} = \left( \frac{2}{e^{x} +1} \right)e^{xt }.$ See \cite[pp. 48]{Com74}.
Then it follows $$\snum(K_{n}; t) = E_n\left(\frac{t+1}{2} \right) 2^n$$ from
$\sum_{n\ge 0} \snum(K_{n}; t) \frac{x^n}{n!} = e^{tx} \sech x 
= \left( \frac{2}{e^{2x} +1} \right)e^{2x\left(\frac{t+1}{2}\right)} = \sum_{n\ge 0} E_n\left(\frac{t+1}{2} \right) \frac{(2x)^n}{n!}.
$\end{rmk}

\begin{thm}
Let $K_{1,n}$ be the \emph{star graph} with $n+1$ vertices, which is a tree with at least one vertex of degree $n$. Then the exponential generating function for signed a-polynomials of $K_{1,n}$ is given by
\begin{align}
\sum_{n\ge 0} \snum(K_{1,n}; t) \frac{x^{n}}{n!} &=  e^{tx} (t-\tanh x).
\label{eq:gfk1n}
\end{align}
\end{thm}

\begin{proof}
From Theorem 2.9 in \cite{CP12}, it is known that
$$\snum (K_{1,2n+1}) =  (-1)^{n+1} A_{2n+1}.$$
Using the formula \eqref{eq:sagt}, we obtain
\begin{align}
\sum_{n\ge 0} \snum(K_{1,n}; t) \frac{x^n}{n!}
&= \sum_{n\ge 0} \left(\snum(\emptyset) t^{n+1} + \sum_{j=0}^{\lfloor n/2 \rfloor} \binom{n}{2j+1} \snum(K_{1,2j+1}) \ t^{n-(2j+1)} \right) \frac{x^n}{n!} \notag \\
&= \sum_{n\ge 0} t^{n+1} \frac{x^n}{n!}  + \sum_{k\ge 0} \sum_{j\ge 0} \binom{k+2j+1}{2j+1} (-1)^{j+1} A_{2j+1} \ t^{k} \frac{x^{k+2j+1}}{(k+2j+1)!} \notag \\
&=  t \left( \sum_{n\ge 0} \frac{(tx)^n}{n!} \right)  +  \left( \sum_{k\ge 0} \frac{(tx)^{k}}{k!} \right)  \left( \imath \sum_{j\ge 0}  A_{2j+1} \frac{(\imath x)^{2j+1}}{(2j+1)!} \right) \label{eq:gfK1}
\end{align}
By \eqref{eq:zigzag}, it follows $$\sum_{j\ge 0}  A_{2j+1} \frac{(\imath x)^{2j+1}}{(2j+1)!} = \tan(\imath x)= \imath \tanh x$$
and the formula \eqref{eq:gfK1} becomes the formula \eqref{eq:gfk1n}.
%
\end{proof}

\begin{table}[t]
\centering
\begin{tabular}{c||c|c}
$G$ & $K_{1,2n-1}$  & $K_{1,2n}$ \\
\hline
$\snum(G)$ & $(-1)^n A_{2n-1}$ & $0$ \\
$\anum(G)$ & $A_{2n-1}$ & $0$\\
$\bnum(G)$ & $0$ & $(-1)^n A_{2n}$ \\
$\cnum_{2i}(G)$
& $\begin{cases} \binom{2n-1}{2i-1} A_{2i-1},& \text{if $i>0$}\\ 1,& \text{if $i=0$} \end{cases}$
& $\begin{cases} \binom{2n}{2i-1} A_{2i-1}  ,& \text{if $i>0$}\\ 1,& \text{if $i=0$} \end{cases}$\\
\hline
\\[-9pt]
e.g.f. for $\snum(G;t)$ & \multicolumn{2}{c}{$\ds \sum_{n\ge 0} \snum(K_{1,n}; t) \frac{x^{n}}{n!} =  e^{tx} (t-\tanh x)$}
\end{tabular}
\caption{Numbers for star graphs $K_{1,n-1}$, where $\sum_{n \ge 0} A_n \frac{z^n}{n!} = \sec z + \tan z$.}
\end{table}

Since $\snum(G) = \snum(G;0)$, putting $t = 0$ in the generating functions \eqref{eq:gfpn}, \eqref{eq:gfcn}, \eqref{eq:gfkn}, and \eqref{eq:gfk1n} yields the generating functions for signed a-numbers of path graphs, cycle graphs, complete graphs, and star graphs as follows:
\begin{align*}
\sum_{n\ge 0} \snum(P_{n}) x^{n} &= \frac{-1+\sqrt{1+4x^2}}{2x^2} = \sum_{m\ge 0} (-1)^m \Cat_m x^{2m} ,\\
\sum_{n\ge 0} \snum(C_{n}) x^{n} &=  \frac{1}{2} + \frac{1}{2\sqrt{1+4x^2}} = 1 + \sum_{m\ge 1} \frac{(-1)^m}{2} \binom{2m}{m} x^{2m} ,\\
\sum_{n\ge 0} \snum(K_{n}) \frac{x^n}{n!} &= \sech x =  \sum_{m\ge 0} (-1)^m A_{2m} \frac{x^{2m}}{(2m)!},\\
\sum_{n\ge 0} \snum(K_{1,n}) \frac{x^{n}}{n!} &=  - \tanh x = \sum_{m\ge 1} (-1)^{m} A_{2m-1} \frac{x^{2m-1}}{(2m-1)!}.
\end{align*}

Similarly, since $\bnum(G) = \snum(G;1)$, putting $t = 1$ in the generating functions \eqref{eq:gfpn}, \eqref{eq:gfcn}, \eqref{eq:gfkn}, and \eqref{eq:gfk1n} yields the generating functions for b-numbers of path graphs, cycle graphs, complete graphs, and star graphs as follows:
\begin{align*}
\sum_{n\ge 0} \bnum(P_{n}) x^{n} &= 1 + \frac{-1 +\sqrt{1+4x^2}}{2x} = 1 + \sum_{m\ge0} (-1)^m \Cat_m x^{2m+1} ,\\
\sum_{n\ge 0} \bnum(C_{n}) x^{n} &=  1 + \frac{x}{\sqrt{1+4x^2}} = 1 + \sum_{m\ge0} (-1)^m \binom{2m}{m} x^{2m+1} ,\\
\sum_{n\ge 0} \bnum(K_{n}) \frac{x^n}{n!} &= 1+ \tanh x = 1 + \sum_{m\ge 0} (-1)^m A_{2m+1} \frac{x^{2m+1}}{(2m+1)!},\\
\sum_{n\ge 0} \bnum(K_{1,n}) \frac{x^{n}}{n!} &=  \sech x = \sum_{m\ge 0} (-1)^m A_{2m} \frac{x^{2m}}{(2m)!}.
\end{align*}


According to \eqref{eq:convert}, putting $t  \leftarrow \frac{1}{\sqrt{-z}}$ and $x \leftarrow x \sqrt{-z}$ in the generating functions \eqref{eq:gfpn}, \eqref{eq:gfcn}, \eqref{eq:gfkn}, and \eqref{eq:gfk1n} yields the next result.
\begin{cor}
Let $P_{M(G)}(z)$ denote the Poincar\'e polynomials of the real toric manifolds $M(G)$ associated to  the graph $G$. Then the generating functions for Poincar\'e polynomials of the real toric manifolds  associated to path graphs $P_n$, cycle graphs $C_n$, complete graphs $K_n$, and star graphs $K_{1,n}$ are as follows:
\begin{align*}
\sum_{n\ge 0} P_{M(P_n)}(z) {x^{n}} &= \frac{-1+2x+ \sqrt{1-4zx^2}}{2x - 2(1+z)x^2},\\
\sum_{n\ge 0} P_{M(C_n)}(z) {x^{n}} &=  \frac{1}{2} + \frac{1}{2\sqrt{1-4zx^2}} \cdot \frac{(1-z) x + \sqrt{1-4zx^2}}{1-(1+z)x},\\
\sum_{n\ge 0} P_{M(K_n)}(z) \frac{x^n}{n!} &= e^{x} \sec (x\sqrt{z})\\
\sum_{n\ge 0} P_{M(K_{1,n})}(z) \frac{x^n}{n!} &= e^{x}\left(1+ \sqrt{z}\tan( x\sqrt{z})  \right).
\end{align*}
\end{cor}

\section{Signed a-number of complete multipartite graphs}
Firstly, we consider the exponential generating function for signed a-numbers of complete bipartite graphs.
Denote by $K_{p,q}$ the \emph{complete bipartite graph} with $p$-set and $q$-set.
\begin{thm}
The exponential generating function for signed a-numbers of complete bipartite graphs is
\begin{align}
\sum_{p \ge 0} \sum_{q \ge 0} \snum(K_{p,q}) \frac{x^p}{p!} \frac{y^q}{q!} = \frac{\cosh x + \cosh y - 1}{\cosh(x+y)}.
\label{eq:snum_bipartite}
\end{align}
\end{thm}

\begin{proof}
For two nonnegative integers $p$ and $q$ whose sum is even, there is the recurrence
\begin{align}
\sum_{i, j\ge 0} {p \choose i } {q \choose j} \snum(K_{i,j}) =
\begin{cases}
0 & \text{if $p$ and $q$ are positive,}\\
1 & \text{if $p$ or $q$ is zero.}
\end{cases}\label{eq:recur}
\end{align}
The exponential generating function for right-hand side of \eqref{eq:recur} is
\begin{align}\label{eq:rhs}
\sum_{{p,q \ge 0} \atop {p+q=\text{even}}} (RHS) ~ \frac{x^p}{p!} \frac{y^q}{q!} = 1 + (\cosh x - 1) + (\cosh y-1).
\end{align}
The exponential generating function for left-hand side of \eqref{eq:recur} is
\begin{align}
\sum_{{p,q \ge 0} \atop {p+q=\text{even}}} (LHS) ~ \frac{x^p}{p!} \frac{y^q}{q!}
&= \sum_{{p,q \ge 0} \atop {p+q=\text{even}}} \sum_{{{0\le i \le p} \atop {0\le j \le q}} \atop {i+j=\text{even}}}
\left(\snum(K_{i,j}) \frac{x^i}{i!} \frac{y^j}{j!}\right) \left( \frac{x^{p-i}}{(p-i)!} \frac{y^{q-j}}{(q-j)!} \right) \notag \\
&= \left(\sum_{{i,j \ge 0} \atop {i+j=\text{even}}} \snum(K_{i,j}) \frac{x^i}{i!} \frac{y^j}{j!}\right)
 \left( \sum_{{i,j \ge 0} \atop {i+j=\text{even}}} \frac{x^i}{i!} \frac{y^j}{j!}\right) \notag \\
&= \left(\sum_{{p,q \ge 0}} \snum(K_{p,q}) \frac{x^p}{p!} \frac{y^q}{q!}\right) \cosh(x+y). \label{eq:lhs}
\end{align}
Thus, by \eqref{eq:rhs} and \eqref{eq:lhs}, we are done.
\end{proof}

The generating function $SA_q(x)$ is defined by
$SA_q(x)= \sum_{p \ge 0} \snum(K_{p,q}) \frac{x^p}{p!},$
which is the coefficient of $y^q/q!$ in $\frac{\cosh x + \cosh y - 1}{\cosh(x+y)}$.
Given a fixed nonnegative $q$, we can induce the detailed formula $SA_q(x)$ by $$SA_q(x) = \left. \frac{\partial^q}{\partial y ^q} \left( \frac{\cosh x + \cosh y - 1}{\cosh(x+y)} \right) \right|_{y=0}.$$
For example, the initial generating functions $A_q(x)$ are listed as follows:
\begin{align*}
SA_0(x) &= 1,\\
SA_1(x) &= - \tanh x,\\
SA_2(x) &= - 2 \sech^2 x + \sech x+1,\\
SA_3(x) &= (6 \sech^2 x - 3 \sech x-1) \tanh x, \\
SA_4(x) &= 24 \sech^4 x -12 \sech^3 x-20\sech^2 x+7 \sech x+1.
\end{align*}

Next, we generalize the generating function \eqref{eq:snum_bipartite} for complete multipartite graphs.
Denote by $K_{p_1,\dots,p_m}$ the \emph{complete $m$-partite graph} with $p_1$-set, \dots, $p_m$-set.
\begin{thm}
The exponential generating function for signed a-numbers of complete $m$-partite graphs is
\begin{align}
\sum_{p_1,\dots, p_m \ge 0} \snum(K_{p_1,\dots ,p_m}) \frac{x_1^{p_1}}{p_1!} \cdots \frac{x_m^{p_m}}{p_m!} = \frac{(1-m) + \cosh x_1 + \dots + \cosh x_m}{\cosh(x_1+\dots +x_m)}.
\label{eq:snum_multipartite}
\end{align}
\end{thm}

\begin{proof}
For $m$ nonnegative integers $p_1, \dots, p_m$ whose sum is even, there is the recurrence
\begin{align}\label{eq:m-recur}
\sum_{i_1,\dots,i_m\ge 0} {p_1 \choose i_1 } \dots {p_m \choose i_m} \snum (K_{i_1,\dots,i_m}) = \begin{cases}
0 & \text{if at least two $p_i$'s are positive,}\\
1 & \text{if all $p_i$'s are zeros, but at most one.}
\end{cases}
\end{align}

Using both sides of \eqref{eq:m-recur}, we have the generalized formulae of \eqref{eq:rhs} and \eqref{eq:lhs} as follows:
$$\sum_{{p_i \ge 0} \atop {p_1+\dots+p_m=\text{even}}} (RHS) ~\frac{x_1^{p_1}}{{p_1}!} \dots \frac{x_m^{p_m}}{{p_m}!}  = 1 +  (\cosh x_1 -1 ) + \dots + (\cosh x_m -1)$$
and
$$\sum_{{p_i \ge 0} \atop {p_1+\dots+p_m=\text{even}}} (LHS) ~\frac{x_1^{p_1}}{{p_1}!} \dots \frac{x_m^{p_m}}{{p_m}!}
= \left( \sum_{p_1,\dots, p_m \ge 0} \snum(K_{p_1,\dots ,p_m}) \frac{x_1^{p_1}}{p_1!} \cdots \frac{x_m^{p_m}}{p_m!} \right)  \cosh(x_1+\dots +x_m),$$
which completes the proof.
\end{proof}

\begin{rmk}
Obviously, the exponential generating functions for a-numbers of complete bipartite graphs and complete $m$-partite graphs are equal to
\begin{align*}
\sum_{p \ge 0} \sum_{q \ge 0} \anum(K_{p,q}) \frac{x^p}{p!} \frac{y^q}{q!} &= \frac{\cos x + \cos y - 1}{\cos(x+y)},\\
\sum_{p_1,\dots, p_m \ge 0} \anum(K_{p_1,\dots ,p_m}) \frac{x_1^{p_1}}{p_1!} \cdots \frac{x_m^{p_m}}{p_m!} &= \frac{(1-m) + \cos x_1 + \dots + \cos x_m}{\cos(x_1+\dots +x_m)}.
\end{align*}
\end{rmk}

\section{Signed a-polynomial of complete multipartite graphs}
Firstly, we consider the exponential generating function for signed a-polynomials of complete bipartite graphs.
\begin{thm}
Let $K_{p,q}$ be the complete bipartite graph with $p$-set and $q$-set. Then the exponential generating function for signed a-polynomials of $K_{p,q}$ is given by
\begin{align}
\sum_{p\ge 0}\sum_{q\ge 0} \snum(K_{p,q}; t) \frac{x^{p}}{p!}\frac{y^q}{q!}
&= e^{t(x+y)} \left( \frac{\cosh x + \cosh y -1}{\cosh(x+y)} \right).
\label{eq:gfkpq}
\end{align}
\end{thm}

\begin{proof}
By definition, we have
\begin{align}
\sum_{p\ge 0 \atop q\ge 0} \snum(K_{p,q}; t) \frac{x^{p}}{p!}\frac{y^q}{q!}
= \sum_{p\ge 0 \atop q\ge 0}\left( \sum_{0 \le p'\le p \atop 0\le q'\le q} \binom{p}{p'}  \binom{q}{q'} \snum(K_{p',q'}) t^{p-p'+q-q'} \right)\frac{x^{p}}{p!}\frac{y^q}{q!}.
\label{eq:1}
\end{align}
Substituting $p'' = p - p'$ and $q'' = q -q'$, the right-hand side of \eqref{eq:1} becomes
\begin{align*}
&\sum_{p''\ge 0 \atop q''\ge 0}\left( \sum_{p' \ge 0\atop q' \ge 0} \binom{p'+p''}{p'}  \binom{q'+q''}{q'} \snum(K_{p',q'}) t^{p''+q''} \right)\frac{x^{p'+p''}}{(p'+p'')!}\frac{y^{q'+q''}}{(q'+q'')!} \\
&= \left( \sum_{p' \ge 0} \sum_{q' \ge 0}  \snum(K_{p',q'}) \frac{x^{p'}}{p'!} \frac{y^{q'}}{q'!} \right)
\left( \sum_{p''\ge 0} \frac{(tx)^{p''}}{p''!} \right) \left(\sum_{q''\ge 0} \frac{(ty)^{q''}}{q''!} \right).
\end{align*}
The formula \eqref{eq:snum_bipartite} completes the proof.
\end{proof}

\begin{rmk}
Since the coefficient of $\frac{y^q}{q!}$ in the formula \eqref{eq:gfkpq} is equal to $\sum_{n\ge 0} \snum(K_{q,n}; t) \frac{x^n}{n!}$,
it holds that
$$ \sum_{n\ge 0} \snum(K_{q,n}; t) \frac{x^n}{n!} = \left. \frac{\partial^q}{\partial y^q} e^{t(x+y)} \left( \frac{\cosh x + \cosh y -1}{\cosh(x+y)} \right) \right|_{y=0}.$$
In case of $q=1$, we have the exponential generating function \eqref{eq:gfk1n} for signed a-polynomials of star graphs again.
\end{rmk}


Similarity, we can deduce the next theorem by the same above method.
\begin{thm}
Let $K_{p_1,\dots, p_m}$ be the complete $m$-partite graph with $p_1$-set, \dots,  $p_m$-set. Then the exponential generating function for signed a-polynomials of $K_{p_1, \dots, p_m}$ is given by
\begin{align}
\sum_{p_1\dots,p_m\ge 0} \snum(K_{p_1,\dots,p_m}; t) \frac{x_1^{p_1}}{p_1!} \dots \frac{x_m^{p_m}}{p_m!}
= e^{t(x_1+\dots+x_m)} \left( \frac{(1-m)+\cosh x_1 + \dots + \cosh x_m}{\cosh(x_1+\dots+x_m)} \right).
\label{eq:gfkpm}
\end{align}
\end{thm}


Since $\snum(G) = \snum(G;0)$, putting $t = 0$ in the generating functions \eqref{eq:gfkpq} and \eqref{eq:gfkpm} gives the two formula
\eqref{eq:snum_bipartite} and \eqref{eq:snum_multipartite}, respectively. Also, since $\bnum(G) = \snum(G;1)$, putting $t = 1$ in the generating functions  \eqref{eq:gfkpq} and \eqref{eq:gfkpm} yields the generating functions for b-numbers of complete bipartite graphs and complete multipartite graphs as follows:
\begin{align*}
\sum_{p\ge 0}\sum_{q\ge 0} \bnum(K_{p,q}) \frac{x^{p}}{p!}\frac{y^q}{q!} &= e^{x+y} \left( \frac{\cosh x + \cosh y -1}{\cosh(x+y)} \right),\\
\sum_{p_1\dots,p_m\ge 0} \bnum(K_{p_1,\dots,p_m}) \frac{x_1^{p_1}}{p_1!} \dots \frac{x_m^{p_m}}{p_m!}
&= e^{x_1+\dots+x_m} \left( \frac{(1-m)+\cosh x_1 + \dots + \cosh x_m}{\cosh(x_1+\dots+x_m)} \right).
\end{align*}

The next result follows from two generating functions \eqref{eq:gfkpq} and \eqref{eq:gfkpm} by plugging in \eqref{eq:convert}.
\begin{cor}
Let $P_{M(K_{p,q})}(z)$ and $P_{M(K_{p_1,\dots,p_m})}(z)$ denote the Poincar\'e polynomials of the real toric manifolds associated to  the complete bipartite graph  $K_{p,q}$ and the complete $m$-partite graph  $K_{p_1,\dots,p_m}$. Then the generating functions for Poincar\'e polynomials $P_{M(K_{p,q})}(z)$ and $P_{M(K_{p_1,\dots,p_m})}(z)$  are equal to
\begin{align*}
\sum_{n\ge 0} P_{M(K_{p,q})}(z) \frac{x^{p}}{p!}  \frac{y^{q}}{q!}
&= e^{x+y} \left( \frac{\cos (x\sqrt{z}) +  \cos (y\sqrt{z}) -1}{\cos(x\sqrt{z}+y\sqrt{z})} \right),\\
\sum_{n\ge 0} P_{M(K_{p_1,\dots,p_m})}(z) \frac{x_1^{p_1}}{p_1!} \dots \frac{x_m^{p_m}}{p_m!}
&= e^{x_1+\dots+x_m} \left( \frac{(1-m)+\cos (x_1\sqrt{z}) + \dots + \cos (x_m\sqrt{z})}{\cos(x_1\sqrt{z}+\dots+x_m\sqrt{z})} \right).
\end{align*}
\end{cor}

Table~\ref{fig:table_s} shows the Poincar\'e polynomials $P_{M(K_{p,q})}(z)$ for $p\le 6$ and $q \le3$.
\begin{table}[t]
\centering
\footnotesize
\begin{tabular}{c|llll}
$p\backslash q$&$0$&$1$&$2$&$3$\\ \hline
$0$& $1$&$1$&$1$&$1$\\
$1$& $1$&$1+z$&$1+2z$&$1+3z+2z^2$\\
$2$& $1$&$1+2z$&$1+4{z}+3z^2$&$1+6{z}+13z^2$\\
$3$& $1$&$1+3z+2z^2$&$1+6{z}+13z^2$&$1+9{z}+39{z}^{2}+31z^3$\\
$4$& $1$&$1+4{z}+8z^2$&$1+8z+34{z}^{2}+27z^3$&$1+12z+86z^2+205z^3$\\
$5$& $1$ & $1+5{z}+20{z}^{2}+16z^3$ & $1+10{z}+70{z}^{2}+167z^3$ & $1+15{z}+160{z}^{2}+763{z}^{3}+617z^4$\\
$6$& $1$ & $1+6z+40{z}^{2}+96z^3$ & $1+12{z}+125{z}^{2}+597{z}^{3}+483z^4$ & $1+18{z}+267{z}^{2}+2123{z}^{3}+5151z^4$
\end {tabular}
\caption{Table for $P_{M(K_{p,q})}(z)$}
\label{fig:table_s}
\end{table}

\section*{Acknowledgment}
This research was supported by Basic Science Research Program through the National Research Foundation of Korea (NRF) funded by the Ministry of Education, Science and Technology (2012-0004476, 2012R1A1A1014154) and INHA UNIVERSITY Research Grant.




\providecommand{\bysame}{\leavevmode\hbox to3em{\hrulefill}\thinspace}
\providecommand{\href}[2]{#2}


\end{document}